\documentclass[12pt]{amsart}

\usepackage{amssymb}
\usepackage{bbm}
\usepackage{amsmath}
\usepackage{amsthm}
\usepackage{datetime}
\usepackage{enumitem}
\usepackage{amsfonts}
\usepackage{hyperref}
\usepackage{cleveref}
\usepackage{comment}
\usepackage{xcolor}
\usepackage{cancel} %\cancel{expression}\xcance{expression}%

\usepackage{graphicx}

\usepackage{layout}

\makeatletter
\@namedef{subjclassname@2020}{%
  \textup{2020} Mathematics Subject Classification}
\makeatother

\usepackage[T1]{fontenc}
\newtheorem{thm}{Theorem}[section]

\newtheorem{lem}[thm]{Lemma}
\newtheorem{prop}[thm]{Proposition}

\newtheorem{cor}[thm]{Corollary}

\newtheorem*{thm*}{Theorem}

\theoremstyle{definition}
\newtheorem{axiom}{Axiom}
\newtheorem{axiomsc}[axiom]{Axiom-Scheme}
\newtheorem*{axiomsc*}{Axiom-Scheme}
\newtheorem{defn}[thm]{Definition}
\newtheorem{term}[thm]{Terminology}
\newtheorem{rem}[thm]{Remark}
\newtheorem{nota}[thm]{Notation}
\newtheorem{conv}[thm]{Convention}

\DeclareMathOperator{\fib}{Fib}
\DeclareMathOperator{\F}{F}

\DeclareMathOperator{\theory}{Th}
\DeclareMathOperator{\tp}{tp}

\newcommand{\NN}{\mathbb{N}}
\newcommand{\ZZ}{\mathbb{Z}}
\newcommand{\ZphiMinus}{\langle \mathbbm{Z}, +, f,0\rangle}
\newcommand{\zphi}{\mathcal{Z}_\varphi^<}
\newcommand{\zphiMinus}{\mathcal{Z}_\varphi}
\newcommand{\Zphi}{\langle \mathbbm{Z},<, +,f,0\rangle}

\newcommand{\fibfloor}[1]{\lfloor #1\rfloor_{\fib}}
\newcommand{\ZphiF}{\langle \mathbbm{Z},<, +,f,\fibfloor{-} ,0\rangle}

\newcommand{\fp}[1]{\left[\varphi #1\right]}

\begin{document}

%\layout{}

\title[Fibonacci Numbers and Beatty Expansions of Presburger Arithmetic]{Fibonacci Numbers and Model-Complete Axiomatization of Presburger Arithmetic Expanded with a Beatty Sequence}

\author[M. Khani]{Mohsen Khani}
\address{Department of Mathematical Sciences\\ Isfahan University of Technology\\ Isfahan \\84156-83111, Iran}
\email{mohsen.khani@iut.ac.ir}
\author[A. N. Valizadeh]{Ali N. Valizadeh}
\address{Department of Mathematics and Statistics\\ University of Isfahan\\Isfahan\\
	81746-73441, Iran + 
	School of Mathematics\\ Institute for Research in Fundamental Sciences (IPM)\\
	P.O. Box 19395-5746, 
	Tehran, Iran   }
\email{a.valizadeh@mcs.ui.ac.ir, valizadeh.ali@ipm.ir}
\author[A. Zarei]{Afshin Zarei}
\address{School of Mathematics\\ Institute for Research in Fundamental Sciences 
	(IPM)\\ P.O. Box: 19395-5746, Tehran\\ Iran.}
\email{a.zarei@ipm.ir}
\thanks{The second author was in part supported by a grant from IPM (No. 1400030021). \\The third author was supported by a grant from IPM}

	\date{\today, \currenttime}

	\begin{abstract}
		
		We introduce a recursive theory that completely axiomatizes the structure $\Zphi$ where $f$ is the function that maps each $x$ to the integer part of $\varphi x $, with $ \varphi$ the golden ratio. We prove that our axiomatization is model-complete in a language expanded with a function to which we refer as the \textit{Fibonacci floor function}. %By proving a constructive version of Kronecker's approximation lemma, we also present an improvement of the axiomatization of the structure $\ZphiMinus$ already studied in \cite{kz}.

	\end{abstract}

\subjclass[2020]{Primary 03B25, 03C10 Secondary 11U09, 11U05 \\ \hspace*{.35cm}2012 \textit{ACM Classification.} 300 - Theory of computation - Logic - Constructive mathematics}


\maketitle

\section{Introduction}

Let $\varphi$ be the golden ratio which, as a quadratic number, is the smallest root of the polynomial $x^2+x+1=0$. Also, let  $f:\mathbbm{Z}\to \mathbbm{Z}$ be the function that maps 
each integer $n$ to the integer part of $\varphi n$, namely 
$\lfloor \varphi n\rfloor $. In this paper we study the structure  $\zphi:=\Zphi$ which may be called a \textit{Beatty expansion} of Presburger arithmetic as the range of $f$ is called the Beatty sequence with modulus $\varphi$. As we proceed, we will see that 
the Fibonacci sequence plays a significant role in determining the model theoretic properties of this structure. 
The Fibonacci sequence is defined recursively as:
\[ \begin{cases}
	F_0=1, F_1=1, \\F_{n+2}=F_{n}+F_{n+1}.
\end{cases} \]
Recall that this sequence is definable in the language of Peano arithmetic which contains addition, multiplication, and order.
We observe that the set of Fibonacci numbers is definable in the language $\mathcal{L}_0=\{<, +, f, 0\}$. Based on that, we introduce a new function $\fibfloor{-}:\mathbbm{Z}\to \mathbbm{Z}$, which we call the \textit{Fibonacci floor function}, that maps each $n$ to the largest Fibonacci number of an even index, and smaller than $n$. As we proceed, we present the required axioms for this function, which, together with the axioms concerning a certain set of extrema, are able to capture the model theoretic content of the structure $\zphi$. This leads to our main theorem (\Cref{thm-model-completeness}):

\begin{thm*}\label{main}
	The structure
	$\ZphiF$ is axiomatizable by a 
	model-complete theory.
\end{thm*}
\noindent 

This theorem provides also a model-theoretic proof for the decidability of the structure $\zphi$.
%The fact that standard integers is the prime model of the introduced theory demonstrates that our axiomatization is complete, which provides also an alternative proof for the decidability of the structure $\zphi$.
A concise history of the subject is provided below.

In a series of works, \cite{Hieronymi-ExpansionByTwoDiscrete,HT,Hieronymi-ScalarMultiplication},  expansions of the ordered additive structure $\langle \mathbbm{R}, \mathbbm{Z}, <, +, 0\rangle$ with various restricted forms of multiplication were explored: By fixing an irrational number $\alpha$, the language is either expanded by a unary function $\lambda_\alpha: \mathbbm{R} \to \mathbbm{R}$, defined as $\lambda_\alpha (x) = \alpha x$, or else it is expanded by predicates of the form $\alpha \mathbbm{Z}$. %either it is added to the language a unary function $\lambda_\alpha: \mathbbm{R} \to \mathbbm{R}$, defined by $\lambda_\alpha (x) = \alpha x$, or else predicates of the form $\alpha \mathbbm{Z}$.

On the one hand, it was shown that the structure $\mathcal{S}_\alpha = \langle \mathbbm{R},\mathbbm{Z}, <, +, \lambda_\alpha\rangle$ is decidable if and only if $\alpha$ is a quadratic number (\cite{Hieronymi-ScalarMultiplication}). On the other hand, it was proved that the structure  $\mathcal{R}_{\alpha} = \langle \mathbbm{R}, \mathbbm{Z}, \alpha\mathbbm{Z}, <, + \rangle$  is decidable whenever $\alpha$ is quadratic (\cite{Hieronymi-ExpansionByTwoDiscrete}). Furthermore, it was shown that having only two scalar multiplications is sufficient for the structure $\langle \mathbbm{R}, \mathbbm{Z},\alpha \mathbbm{Z}, \beta \mathbbm{Z}, <,+ \rangle$ to be undecidable whenever $1$, $\alpha$ and $\beta$ are linearly independent over $\mathbbm{Q}$ (\cite{HT}).

In \cite{kvz} we developed a deep investigation of the unordered integer part of the structure $\mathcal{R}_\alpha$, where we introduced a model-complete axiomatization leading to the result that the structure $\langle \mathbbm{Z}, +, f_\alpha, 0\rangle$ is decidable if and only if $\alpha $ is a computable real number; recall that the function $f_\alpha$ maps an integer $x$ to $\lfloor\alpha x\rfloor$, and it is called the Beatty sequence with modulus $\alpha$.

%There are two main questions in this context: Is the structure $\mathcal{R}_\alpha$ undecidable for a non-quadratic $\alpha$? Recall that, $\mathcal{R}_\alpha$ is a reduct of the structure $\mathcal{S}_\alpha$ which is undecidable for a non-quadratic $\alpha$. Next, given the fact that using techniques from the theory of automata constitutes the main theme in proving the mentioned decidability/undecidability results, another question would be if it is possible to provide model-theoretic proofs for the available decidability results; an affirmative answer to this question leads to a clear logical axiomatization for these structures. 

%We addressed the first question in \cite{kvz}, where we developed a deep investigation of the unordered integer part of the structure $\mathcal{R}_\alpha$ in the general case of an irrational $\alpha$. We introduced a model-complete axiomatization, leading to the result that the structure $\langle \mathbbm{Z}, +, f_\alpha, 0\rangle$ is decidable if and only if $\alpha $ is a computable real number; recall that the function $f_\alpha$ maps an integer $x$ to $\lfloor\alpha x\rfloor$, and it is called the Beatty sequence with modulus $\alpha$.

In the current work we study the structure $\zphi$ that consists of the integer part of the structure $\mathcal{R}_\alpha$ for the case of $\alpha=\varphi$. It is astonishing that having the usual order of integers along with the function $f_\alpha$ leads to an unexpected degree of complication. In fact, extracting the model-theoretic content and finding the right axiomatization of this structure turned into a difficult challenge for a quadratic number like $\varphi$, and this is even more complicated for an arbitrary computable real number. We denote by $\zphiMinus$ the reduct of $\zphi$ that is obtained by forgetting the usual order of integers. This structure was studied in \cite{kz}, where a quantifier elimination was proved. Our works here, also leads to proving a constructive version of Kronecker's approximation lemma which provide a significant improvement of the axiomatization of the structure  $\zphiMinus$; this will appear in \Cref{sec-without-num-order}.
\par  
%
%In \Cref{sec-without-num-order}, we prove \Cref{thm-DLO} which implies a significant improvement for the axiomatization of the structure $\zphiMinus$ by showing that a constructive version of Kronecker's approximation lemma is a theorem of the simpler axiomatization presented here. We recall that Kronecker's approximation lemma constitutes the focal apparatus in proving quantifier elimination for $\zphiMinus$.
%As we will discuss in \Cref{sec-without-num-order}, we briefly outline the main ideas applied to prove quantifier-elimination for the structure $\zphiMinus$ while trying to encounter the main challenges that arise by adding the usual order of integers. %In our way towards introducing a model-complete theory for $\zphi$, these challenges will be highlighted in \Cref{sec-restoring-order}.

We introduce our axiomatization for the structure $\zphi$ in \Cref{sec-restoring-order}. As we proceed, the importance of Fibonacci numbers and their connection to certain extremum points will become clearer.
%We will continue in \Cref{sec-restoring-order} by proving several lemmas and corollaries to introduce our axiomatization for the structure $\zphi$ step by step.
We will conclude with some final observations and remarks in the last section. 
%the presented details will gradually illuminate the importance of Fibonacci numbers and their definability, and moreover, the importance of treating certain set of extrema in this structure. We will close by some final observations and remarks in the last section. 

\vspace*{.5cm}
\textbf{A‌ remark on terminology.} To make the exposition clearer, we will frequently use the term ``numerical order'' in contrast to the term ``decimal order'' of integers. While the former simply refers to the usual ordering of integers, the meaning of the latter will be discussed and introduced in \Cref{sec-without-num-order}.

\begin{conv}
 Unless stated otherwise, by definable we mean $\varnothing$-definable.
\end{conv}

\vspace*{.7cm}
\section {In the absence of numerical order, the structure $\zphiMinus$}\label{sec-without-num-order}
\vspace*{.3cm}

As addressed above, an axiomatization of the structure 
$\zphiMinus:=\ZphiMinus$ together with a quantifier elimination were introduced in \cite{kz}. In this brief section we show that the set of basic axioms for the function $f$ are far more strong than what was expected. In fact, a constructive version of Kronecker's approximation lemma turns out to be an implication of these basic axioms.

%In \Cref{thm-DLO} below, we show that any expansion of the additive structure of integers with a function satisfying \Cref{axiom-properties-of-f} defines a dense linear order. %As will be discussed below, this axiom scheme will also imply a constructive version of Kronecker's approximation lemma for the structure $\zphiMinus$. 

\begin{axiomsc}[Basic properties]\hfill 
	\label{axiom-properties-of-f}
	\begin{enumerate}
		\item [(A1.1)] The axioms of the structure $\langle\ZZ,+,-,0\rangle$ as a $\ZZ$-group.
		\item [(A1.2)] For all $x$ and $y$, either $f(x+y)=f(x)+f(y)$ or $f(x+y)=f(x)+f(y)+1$.
		\item [(A1.3)] For all $x$, the value of $f(-x)$ is equal to $-f(x)-1$.
		\item [(A1.4)] For all $x$, we have that  $f^2(x)=f(x)+x-1 $ and $f(f(x)+x)=2f(x)+x$.
		%			\item [(A1.5)] $f(0)=0$.
		\item [(A1.5)] For all $x$, there is $y$ such that either we have $x=f(y)$ or we have $x = f(y)+y$ (the disjunction is exclusive).
		
		%			\item  ``$[\varphi x]<[\varphi y]$'' is a linear order.
	\end{enumerate}
\end{axiomsc}
%	It is worth noting that (A1.5) states that the set of integers is partitioned by the range of $f$ and $f+\id$. 
Axiom (A1.2) might be called the ``approximate linearity with a uniform bound'' of the function $f$. %Axiom (A1.4) implies that any term in the language reduces to a linear combination of the form $mx+nf(x)$ for some $m,n\in\ZZ$.
Axiom (A1.5) shows that the set of integers is partitioned by the ranges of the function $f(x)$ and the term $x+f(x)$. 
\begin{nota}\label{not-decimal-order}\hfill
	(1) Let $x<^*y$ denote the formula $x\neq y \wedge f(y-x)=f(y)-f(x)$. This relation is called the \textit{decimal order} over integers. (2) For convenience, we let $\bar{f}(x)$ denote the term $x+f(x)$.  
	
\end{nota}
%The discussion following \Cref{crl-SOP} will show why $<^*$ is called thus.

%	\begin{lem}\label{lem-LO}
	%		 Assuming parts (1)-(3) of \Cref{axiom-properties-of-f}, the relation $<^*$ defines a liner order compatible with addition, i.e, for all $x, y$ and $z$ it satisfies $x+z<^*y+z$ whenever $x<^*y$.
	%	\end{lem}
%	\begin{proof}
	%		
	%		
	%	\end{proof}

\begin{lem}\label{lem-LO}
	Assuming parts (1)-(3) of \Cref{axiom-properties-of-f}, 
	\begin{itemize}
		\item[(1)] The relation $<^*$ defines a linear order.
		\item[(2)] For all nonzero elements $x, y$ and $z$, the inequality $x<^*y$ implies $x+z<^*y+z$ whenever $z<^*-y$ or $-x<^*z$.
	\end{itemize}
	
\end{lem}
\begin{proof}
	(1) Irreflexivity is obvious. Suppose that $x<^*y$, we show that $y<^*x$ cannot hold. By (A1.3), we have that $f(x-y)=f(-(y-x))=f(y-x)-1$. Using $x<^*y$, we have that $f(x-y)=f(y)-f(x)-1$ which shows that $y\not<^* x$. To show that $<^*$ is transitive, suppose that $x<^*y$ and $y<^*z$. Using (A1.2), note that $f(z-x)=f(z-y+y-x)=f(z-y)+f(y-x)+i = f(z)-f(x)+i$ where $i$ is either $0$ or $1$. Using (A1.2) and (A1.3), one can easily show that $f(z-x)=f(z)-f(x)+j$ for some $j\in\{-1,0\}$. Hence, $i $ remains to be zero, that is, $f(z-x)=f(z)-f(x)$ which implies that $x<^*z$.

	Now, suppose that $x\neq y$ and $x\not<^* y$. By (A1.2) and (A1.3) we have that $f(y-x)=f(y)-f(x)-1$. By (A1.3), we have $f(x-y)=f(-(y-x))=-f(y-x)-1 = f(x)-f(y)+1-1=f(x)-f(y)$ which shows that $y<^*x$.
	
	(2) Suppose $x<^* y$ and note that (A.1.3) directly implies $-y<^*-x$. If $z<^*-y$, using transitivity of $<^*$, we have that $z<^*-x$ also. By \Cref{not-decimal-order} and (A1.3), it means that 
	\begin{align}\label{eq-f-x-z}
		f(x+z)=f(x)+f(z) \qquad \text{ and }\qquad  f(y+z)=f(y)+f(z).
	\end{align}
	If $x+z\not<^*y+z$, we can use \Cref{not-decimal-order}, (A1.2), and (A1.3) to conclude that $f((y+z)-(x+z))=f(y+z)-f(x+z)-1$. Using Equations \eqref{eq-f-x-z}, the latter is equal to $f(y)-f(x)-1$. On the other hand, $f((y+z)-(x+z))=f(y-x)$ which, by our assumption of $x<^*y$ and \Cref{not-decimal-order}, is equal to $f(y)-f(x)$, which is a contradiction.
	
	In the case that $-x<^*z$, we have $f(x+z)=f(x)+f(z)+1$ and $f(y+z)=f(y)+f(z)+1$ instead of Equations \eqref{eq-f-x-z}, which can be used in a similar argument as above to show that $x+z<^*y+z$.
	
	%		By (A1.2), we have that $f(x+z)=f(x+y-y+z)=f(x-y)+f(y+z)+i$ for some $i\in{0,1}$. By (A1.3) and \Cref{eq-f-x-z}, this means that $f(x+z)=-f(y-x)-1+f(y)+f(z)+i$. Using \Cref{not-decimal-order} and the fact that $x<^*y$, we have $f(y-x)=f(y)=f(x)$. Hence, we have that $f(x+z)$
\end{proof}

\begin{rem}
	Despite the cases treated in part (2) of \Cref{lem-LO}, the relation $ <^*$ is not fully compatible with addition and subtraction. For example, we cannot conclude $x+z<^*y+z$ from $x<^*y$ in cases where $-y<^*z<^*-x$. Also, $0<^*x$ does not imply $-x<^*0$. And, for any nonzero $x$, we have $0<^*x$. Although it seems to be quite an odd situation, it is enough for $<^*$ to simply be an order without having a strong bond to the operations of addition and subtraction. Nevertheless, somewhat desirable properties hold for $<^*$. For example, using (A1.3), for all nonzero elements $x$ and $y$, we can conclude $-y<^*-x$ from $x<^*y$.
\end{rem}

\begin{thm}\label{thm-DLO}
	Assuming parts (1)-(4) of \Cref{axiom-properties-of-f}, the relation $<^*$ defines a dense liner order.
\end{thm}
\begin{proof}
	
	In \Cref{lem-LO} we proved that $<^*$ is a linear order. Suppose that $y<^*z$, we show that the element $z+f(z-y)$ satisfies $ y<^*z+f(z-y)<^*z$. 
	
	\vspace*{.3cm}
	\textbf{Claim 1.} For all $x$ we have $\bar{f}(x)=f(x)+x<^*x$. 
	
	For all $x$ we have that $f(x-(f(x)+x))=f(-f(x))$ which, by (A1.3), is equal to $-f^2(x)-1$. On the other hand, by (A1.4), we have that $f(f(x)+x)=f^2(x)+f(x)+1$, and hence, we have $f(x)-f(f(x)+x)=-f^2(x)-1$. This, using \Cref{not-decimal-order}, shows that $f(x)+x<^*x$ for all $x$.
	
	Using Claim 1 for $z-y$, we have that 
	$f(z-y)+z-y<^*z-y$. 
	Therefore, it suffices to show that $y<^*y-z$ and then to use part (2) of \Cref{lem-LO} to obtain the desired inequality. 
	Suppose that $y\not<^*y-z$. By \Cref{not-decimal-order}, (A1.2), and (A1.3), this means that $f(y-z-y)=f(y-z)-f(y)-1$. 
	Using the assumption $y<^*z$, by \Cref{not-decimal-order}, (A1.2), and (A1.3), the latter expression is equal to $f(y)-f(z)-1-f(y)-1 = -f(z)-2$. 	
	On the other hand, $f(y-z-y)=f(-z)$ which, by (A1.3), is equal to $-f(z)-1$, which is a contradiction. 
	
	To show that $y<^*z+f(z-y)$, we first prove a claim. 
	
	\vspace*{.3cm}
	\textbf{Claim2.} $f(y)\not<^*z+f(z)=\bar{f}(z)$. 
	
	\textit{proof of the claim.} Otherwise, by \Cref{not-decimal-order}, we have that $f(z+f(z)-f(y))=f(z+f(z))-f^2(y)$ which, by (A1.4), is equal to $2f(d)+d-f(c)-c+1$. Using the assumption $y<^*z $, by \Cref{not-decimal-order}, we have that $f(z-y)=f(z)-f(y)$. hence, we have that $f(z+f(z)-f(c))=f(z+f(z-y))$ which, by (A1.2), is equal to $f(d)+f^2(z-y)+i$  where $i$ is either $0$ or $1$. Using (A1.4), the latter is equal to $f(z)+f(z-y)+z-y-1+i$ which, using the assumption $y<^*z$ again, equals $2f(z)+d-f(y)-y+j$ for some $j\in\{-1,0\}$ which is a contradiction and proves the claim. 
	
	Now, suppose that $y\not<^*z+f(z-y)$. This implies that $f(f(z-y)+z-y) = f(f(z-y)+z)-f(y)-1$. Using the assumption of $y<^*z$, the latter is equal to $f(f(z)+z-f(y))-f(y)-1$ which, by Claim 2, is equal to $f(f(z)+z)-f^2(y)-1-f(y)-1$. This, by (A1.4), is equal to $2f(z)+z-f(y)-y+1-1-f(y)-1$ or $ 2f(z)-2f(y)+z-y-1$.
	
	On the hand, by (A1.4) we have that $f(f(z-y)+z-y))=2f(z-y)+z-y$ which, using the assumption of $y<^*z$, is equal to $2f(z)-2f(y)+z-y$. This contradicts the outcome of the last paragraph.
	
\end{proof}

A direct consequence of this theorem is: 
\begin{cor}\label{crl-SOP}
	Any expansion of the additive structure of integers with a unary function satisfying parts (1)-(4) of \Cref{axiom-properties-of-f} has the strict order property.
\end{cor}

Although the structure $\zphiMinus$ deals with integers, it was observed in \cite{kz} and \cite{kvz} that the order relation over a countable subset of the ``real'' interval $(0,1)$ is definable within this structure. More precisely, considering the set of all decimal parts $[\varphi x]:=\varphi x-\lfloor \varphi x\rfloor$, it is easy to observe that for all integers $x$ and $y$ the relation $[\varphi x]<[\varphi y]$ coincides with the order $x<^*y$ defined above (see \Cref{not-decimal-order} again).

%We find it useful to briefly address the essential aspects of the challenges of eliminating quantifiers for formulas in the structure $\zphiMinus$; this  will also prepare us to better confront the challenges that will be treated in the next section. 
%
%Due to the fact that $\varphi$ is a quadratic number, for all $x$ we have $f^2(x)=f(x)+x-1$, as appeared in \Cref{axiom-properties-of-f}. Consequently, in a model of \Cref{axiom-properties-of-f} all terms are reduced to linear terms of the form $m_1x+m_2f(x)$ for some integer coefficients $m_1$ and $m_2$. As addressed above, the relation $[\varphi x]<[\varphi y]$ is definable, and, if we allow this relation in the language all quantifier-free formulas are reduced to even a simpler form. Therefore, and very concisely put, in the absence of the numerical order, one must essentially address the following question to prove elimination of quantifiers for the structure $\zphiMinus$:
%
%\begin{quote}
%	($\ast$) Assume that $M_1,M_2$ are two models, $M$ is a common substructure, and $c,d$ are in $M$. 
%	If there is $x\in M_1$ such that 
%	$[\varphi c]<[\varphi x]<[\varphi d]$,  is there any guarantee  
%	that such an element exists also in  $M_2$?~\footnote{
%		More accurately, in the structure $\zphiMinus$, the formulas are actually of the form $(r+s\varphi )[\varphi c]<[\varphi x]<(r'+s'\varphi )[\varphi d]$, which calls for a careful treatment (see \cite{kz}).
%	}
%\end{quote}
The one-dimensional Kronecker's lemma states that the set $\{[\varphi n]: n\in\ZZ\}$ is dense in the unit interval $(0,1)$. This theorem in crucial for proving the model-theoretic properties of the structure $ \zphiMinus $. \Cref{thm-DLO} shows that a first-order version of Kronecker's lemma is encoded in the first-order theory of $\zphiMinus$, namely: 

\begin{align}\label{eq-Kronecker}
\forall y\forall z \Big([\varphi y]<[\varphi z]\rightarrow \exists x\big([\varphi y]<[\varphi x]<[\varphi z]\big)\Big).	
\end{align}

That is, not only the sentence above is true in 
$\zphiMinus$, but it surprisingly turns out to be provable from the axioms describing the basic properties of the function $f$. In fact, the proof of \Cref{thm-DLO} provides a constructive way to find the desired element in the one-dimensional Kronecker's lemma:

\begin{cor}[Constructive Kronecker's Lemma]\label{crl-Kronecker}\hfill
	
Assuming (1)-(4) of \Cref{axiom-properties-of-f}, we have
\begin{align}\label{eq-Kronecker-const}
 [\varphi x]<[\varphi\big(f(y-x)+y\big)]<[\varphi y]
\end{align}
for all $x$ and $y$ satisfying $\fp{x}<\fp{y}$.
\end{cor}
%
%As shown in \cite{kz}, Kronecker's approximation lemma is crucial to positively answer question ($\ast $) mentioned above. As shown in \Cref{crl-Kronecker}, a constructive form of this lemma, namely Formula \eqref{eq-Kronecker-const}, is now a theorem of parts (1)-(4) of \Cref{axiom-properties-of-f}, improving the axiomatization introduced in \cite{kz}.

\vspace*{.7cm}
\section{Restoring the numerical order, the structure $\zphi$}\label{sec-restoring-order}
\vspace*{.3cm}
	
Our aim is to present a model-complete theory that completely axiomatizes the structure $\zphi=\Zphi$. Only to make the context more clear, we will freely be using the subtraction function without adding it to the language. 

Let $\mathcal{L}=\{<,+,f,0\}$. We first strengthen \Cref{axiom-properties-of-f} by replacing (A1.1) with the following axiom scheme which, in particular, describes the expected properties of integers as an ordered group:

\begin{axiomsc*}\hfill
	
	\noindent $\text{(A1.1)}'$ \quad The axioms of Presburger arithmetic expressed in the language $\{<,+,0\}$.
\end{axiomsc*}

When both numerical and decimal orders are present in the structure, one needs to deal with extra complications which cannot be handled solely based on Kronecker's approximation lemma, and which make it necessary to have a more precise understanding of the distribution of decimal parts $[\varphi x]$  when $x$ ranges over a numerical interval like $(a,b)$. We will see that the decimal parts corresponding to the elements of Fibonacci sequence carry much of crucial information needed to understand this distribution. In addition, the extremum points which appear through the interplay of the existing orders need to be treated carefully.

By the following axiom, we guarantee the existence of such extrema; a fact that is already true in $\zphi$ since there are at most finitely many integers between two given ones:
\begin{axiom}\label{ax-min-max}
	For all $x<y$ there are $z_1 $ and $z_2$ such that $x<z_1<y $, $x<z_2<y$, and 
	\begin{align*}
			[\varphi z_1]&=\min \big\{[\varphi w]: x<w<y\big\},\\
			[\varphi z_2]&=\max \big\{[\varphi w]: x<w<y\big\}.
	\end{align*}
\end{axiom}
The following lemma shows that such extrema over intervals of the form $(0,b)$ are actually the Fibonacci numbers of the standard model:

\begin{lem}\label{lma-fib-min-max}
	Working in $\zphi$, suppose that $a\in\NN$. Then, $a$ is a Fibonacci number of an even (odd) index if and only if $\fp{a}$ is the minimum (maximum) of the following set: 
	\[ \Big\{\fp{n}\ \Big|\ 0< n\leq a\Big\}. \]
	%all decimal parts $\fpp{n}$ among all positive integers $n$ less than or equal to $a$.
\end{lem}
\begin{proof}
	We claim that for any Fibonacci number of an even index, say $F_{2n}$, we have that $f(F_{2n}+x)=f(F_{2n})+f(x)$ for all $0<x<f(F_{2n})=F_{2n+1}$. This claim asserts that $\fp{F_{2n}}<\fp{y}$ for all $F_{2n}<y<F_{2n+2}$. Having the claim, we can use the fact that $\fp{F_{2n+2}}<\fp{F_{2n}}$ to finish the proof. 
	\par 
	%Suppose that there exists natural number $m$, such that  $0<m<F_{2n+1}$ and $f(F_{2n}+m)=f(F_{2n})+f(m)+1$.
	By Fact 4 in  \cite{kz}, for all natural number $n$ and any integer $y$ satisfying $0<y\leq F_{2n}$ we have that $f(F_{2n}+y)=f(F_{2n})+f(y)$. Therefore, it suffices to  prove the claim for all  integer $x$ satisfying $F_{2n}<x<F_{2n+1}$.
	
	For such an $x$, there exists $0<y<F_{2n-1}$ such that $x=F_{2n}+y$. Hence, we have that
	\begin{align*}
		f(F_{2n}+x)=f(F_{2n}+(F_{2n}+y))=f(2F_{2n}+y)=f(F_{2n+1}+F_{2n-2}+y).
	\end{align*}
	Since $y<F_{2n-1}$ and $F_{2n-2}+y<F_{2n}$, we can use Fact 4 in \cite{kz} to show that
	$f(F_{2n+1}+(F_{2n-2}+y))=f(F_{2n+1})+f(F_{2n-2}+y)+1$. On the other hand, as $y<F_{2n}$ we have that $f(F_{2n-2}+y)=f(F_{2n-2})+f(y) $. Therefore, we have that
	\begin{align*}
		f(F_{2n}+x)&=f(F_{2n+1}+(F_{2n-2}+y))\\ & = f(F_{2n+1})+f(F_{2n-2}+y) +1\\
		& = f(F_{2n+1})+f(F_{2n-2})+f(y) +1\\
		&=f(F_{2n+1}+F_{2n-2})+f(y)\\& =f(2F_{2n})+f(y)\\
		&=f(F_{2n})+f(F_{2n})+f(y)\\&=f(F_{2n})+f(F_{2n}+y)\\&=f(F_{2n})+f(x).
	\end{align*}
	In a similar way, one can prove that the maximum of the set specified in the lemma occurs at a Fibonacci number with odd index.
\end{proof}

\begin{comment}

Fixing some terminology as below makes explaining the rest of the work much easier:
\begin{term}
	By the decimal part of an element $x$, we mean
	$[\varphi x]$. So it will be clear when we say that
	$x$ is an element whose decimal part is smaller than that of $y$.
	Also we may say that $x$ falls in the numerical interval
	$(a,b)$ and decimal interval $(c,d)$, which means that
	$a<x<b$ and $[\varphi c]<[\varphi x]<[\varphi d]$.
\end{term}
\end{comment}

The lemma above has the following consequence whose advantage may be better observed when contrasted with the fact that Fibonacci sequence--by being a recursive sequence--is definable in the language of Peano Arithmetic, which, in particular, contains multiplication:
\begin{cor}\label{crl-fib-definalbe}
	The set of Fibonacci numbers is  $\Pi_1$-definable in 
	$\zphi $ by an $ \mathcal{L} $-formula.
	
\end{cor}

%Let $\fib(x)$ be the formula that defines 
%the Fibonacci numbers in $\mathcal{L}$. 

\begin{rem}\label{rem-fib-hull}
Having a formula--in a language not containing multiplication--that defines Fibonacci numbers in $\zphi$ enables one to talk about \textit{non-standard Fibonacci elements} in an arbitrarily given model of $\theory(\zphi)$. In fact, in any such model $ M $ one can speak of \textit{Fibonacci hull} of the model, which can be shown to be existentially closed in $ M $.	
\end{rem}

Using the corollary above, we can define the following functions in $\mathcal{L}$:

\begin{defn}\label{dfn-F-G}
	Let $F(x)$ and $G(x)$ denote the following definable functions:
	\begin{align}\label{eq-F-G}
		\begin{split}
		F(x)=y \ &\Leftrightarrow \ 0<y\leq x \ \wedge \ [\varphi y]=\min \{[\varphi w]: 0<w\leq x\}, \ \text{ and }\\
		G(x)=y \ &\Leftrightarrow \ 0<y\leq x \ \wedge \ [\varphi y]=\max \{[\varphi w]: 0<w\leq x\}.
		\end{split}
	\end{align}
 \end{defn}

\begin{lem}\label{lma-FibFunctions-Properties}
	The following sentences are true in $\zphi$:
	\begin{itemize}
		\item[(1)] For all $x$, the first Fibonacci number of an even index and strictly greater than $x$ is $\bar{f}(F(x))$.
		\item[(2)] For all $x$, if $G(x)<F(x)$, then the first Fibonacci number of an odd index and strictly greater than $x$ is $f(F(x))$. 
		\item[(3)] For all $x$, if $F(x)<G(x)$, then the first Fibonacci number of an odd index and strictly greater than $x$ is $f(\bar{f}(F(x)))$, or equally, $f^{-1}(F(x)-1)$.
	\end{itemize}
	As a consequence, the following hold in $\zphi$:
	\begin{align*}
		&\forall x \forall y\big( F(x)<y<\bar{f}(F(x))\rightarrow F(y)=F(x)\big),\\
		&\forall x \Big(G(x)<F(x)\rightarrow \forall y \big(G(x)<y<f(F(x))\rightarrow G(y)=G(x)\big)\Big),\\
		&\forall x \Big(F(x)<G(x)\rightarrow \forall y \big(G(x)<y<f(\bar{f}(F(x)))\rightarrow G(y)=G(x)\big)\Big).
		\end{align*}
\end{lem}
\begin{proof}
	Easy to verify. %For parts ?? see Lemma ?? in \cite{kz}; the remaining parts are easy to verify. 
\end{proof}
%\begin{rem}\label{rem-G-from-F}
%	Note that $G(x)=f(F(x))$ or $G(x)=f^{-1}(F(x)-1)$, see \cite{kz}, lemma ??. 
%	Also, the first Fibonacci element with an even index and greater than $x$ is 
%	$f(F(x))+F(x)$.
%\end{rem}

\begin{nota}\label{not-fib-floor}
As $F(x)$ maps each integer $x$ to the greatest Fibonacci number of an even index and less than or equal to $x$, a more telling notation for $F$ would be $\fibfloor{-}$ which may be called the \textit{Fibonacci floor function}. We do not need a similar notation for Fibonacci numbers of an odd index, since the function $G$ can be defined using the functions $f $‌ and $ F$ as addressed in \Cref{lma-FibFunctions-Properties}. However, for the sake of better readability we postpone applying this notation until the end of this section and will keep using the symbols $F$ and $G$ until then; the only exception is part (1) of \Cref{not-language}.
\end{nota}

\begin{axiom}\label{ax-Fib-Functions}
	The functions $ F $ and $ G $ satisfy Formulas \eqref{eq-F-G} in \Cref{dfn-F-G}, and 
	parts (1) to (3) of \Cref{lma-FibFunctions-Properties}.
\end{axiom}
Returning to the more general situation of a numerical interval $(a,b)$, there is actually an algorithm in $\zphi$
to find the element of the smallest corresponding decimal part among all integers between $a$ and $b$. However, this approach cannot fully be applied to a non-standard model (see \Cref{subsec-algorithm}). We present the mentioned algorithm in \Cref{prop-algorithm} which can be interesting from a computational point of view.

Alternatively, we will see that the minimum and maximum points of certain formulas will carry the information needed to be preserved from a model to its extensions. First, we add the functions $ F $, formally denoted by $\fibfloor{-}$, to the language:

\begin{nota}\label{not-language}\hfil
	\begin{enumerate}

	\item 
	Let $ \mathcal{L}_{\fib} $ denote the language $ \mathcal{L}\cup\big\{\fibfloor{-}\big\} $ (see \Cref{not-fib-floor}).
	
	\item 
	In a structure $ M $ satisfying \Cref{ax-min-max}, for any elements $ a, b\in M $, let $ \min_{(a,b)}^{M} $ and $ \max_{(a,b)}^{M} $ respectively denote the elements $ u,v\in(a,b) $ with
	\begin{align*}
		[\varphi u]&=\min \big\{[\varphi t]: a<t<b\big\},\\
		[\varphi v]&=\max \big\{[\varphi t]: a<t<b\big\}.
	\end{align*}
	\end{enumerate}
\end{nota} 

From now on, we will be working in the language $\mathcal{L}_{\fib}$. Note that, by \Cref{lma-FibFunctions-Properties}, the function $G$ is definable using $f$ and $F$.
% 
%\begin{axiom}\label{ax-F-G}
%	The functions $ F $ and $ G $ satisfy Formulas \eqref{eq-F-G} in \Cref{dfn-F-G}.
%\end{axiom}

\newpage 
\begin{lem}[\textbf{Main Lemma}]\label{lma-Main}\hfill 
	
	Assume that $M_1\subseteq M_2$ are two models of \Cref{axiom-properties-of-f}, \Cref{ax-min-max}, and \Cref{ax-Fib-Functions} with
	$a,b\in M_1$, 
	then $\min_{(a,b)}^{M_1}=\min_{(a,b)}^{M_2}$.
\end{lem}
\begin{proof}
	Assume that $c=\min_{(a,b)}^{M_1}$ and 
	$d=\min_{(a,b)}^{M_2}$. 
	First, observe that $ c=\min^{M_1}_{[c,b)} $. Hence, for all $ x\in (0, b-c)^{M_1} $ we have that $ \fp{(c+x)}=\fp{c}+\fp{x} $. Otherwise, we have that 
	$ \fp{(c+x)}=\fp{c}+\fp{x}-1<\fp{c} $. Similarly, observe that $ c=\min^{M_1}_{(a,c]} $, and hence, for all $ y\in (0, c-a)^{M_1} $ we have that $ \fp{(c-y)}=\fp{c}-\fp{y}+1 $.
	 
	\par \noindent
	\textbf{Case 1.}  $d>c$: There exists $ u\in (0,b-c)^{M_2} $ such that $ d=c+u $. Obviously, $ \fp{d}<\fp{c} $, hence, 
	$ \fp{(c+u)} $ is forced to be equal to $ \fp{c}+\fp{u}-1 $. That is, $ \fp{u}>\fp{c}-1 $. By \Cref{dfn-F-G},  we have that $ \fp{G(b-c)}\geq \fp{u} $ which 
	contradicts the discussion above for $ x=G(b-c)$. 
	
	\begin{comment}

	we have that 
	$ \fp{c+G(b-1)} = \fp{c}+\fp{G(b-1)}-1 $.
	
	Let's assume that $d=c+x$. 
	In this case there is
	$x\in (0,b-c)\cap M_2$ such that $[\varphi x]+[\varphi c]>1$,
	but there is no $x\in (0,b-c)\cap M_1$ with this property.
	Hence
	$\max \{[\varphi x]: x\in (0,b-c)\cap M_1\}<1-[\varphi c]$
	but 
	$\max \{[\varphi x]: x\in (0,b-c)\cap M_2\}>1-[\varphi c]$.
	This contradicts the fact that the largest  odd Fibonacci before $b-c$ is the same 
	$M_1$ and $M_2$.
	\end{comment}
	
	\par \noindent
	\textbf{Case 2.}  $d<c$: There exists $ u\in (0,c-a)^{M_2} $ such that $ d=c-u $. Obviously, $ \fp{d}<\fp{c} $, hence, 
	$ \fp{(c-u)} $ is forced to be equal to $ \fp{c}-\fp{u} $. That is, $ \fp{u}<\fp{c} $. By \Cref{dfn-F-G},  we have that $ \fp{F(c-a)}\leq \fp{u} $ which 
	contradicts the discussion before Case 1 for $ y=F(c-a)$.
	
	\begin{comment}
		Let's assume that $d=c-x$.
	So, in this case there is $x\in (0,c-a)\cap M_2$ such that
	$[\varphi x]<[\varphi c]$, but there is
	no $x\in (0,c-a)\cap M_1$ with this property.
	That is
	$\min [\varphi x]: x\in (0,c-a)\cap M_2<[\varphi c]$
	while $\min \{[\varphi x]: x\in (0,c-a)\cap M_1>[\varphi c]\}$.
	This contradicts the fact the largest even Fibonacci before 
	$c-a$ is the same in $M_1$ and $M_2$.
	\end{comment}

\end{proof}

\begin{cor}
	Let $M_1\subseteq M_2$ be as in \Cref{lma-Main}, 
	and $a,b,c\in M_1$. If there exists an element $ x $ in 
	$ M_2 $ such that $ a<x<b $ and $ \fp{x}<\fp{c} $, then $ M_1 $ 
	contains also an element satisfying the same properties.
\end{cor}
\begin{proof}
	Existence of $x\in (a,b)$ with $[\varphi x]<[\varphi c]$ is
	equivalent to the fact that
	$\fp{\min_{(a,b)}^{M_2}}<[\varphi c]$. But,  by \Cref{lma-Main}, $ \min_{(a,b)}^{M_2} = \min_{(a,b)}^{M_1} $ belongs to $ M_1 $. 
\end{proof}

We require the corollary above to hold in a wider context, but this needs an extra axiom. Note that the minimum and maximum of the solutions of any definable set intersected with a bounded numerical interval always exist in the standard model, a fact that validates the following axiom:
\begin{axiom}\label{ax-min-decimal-greater-than-phi-c}
For all $u_1,u_2,v_1,v_2$, if there is $y\in (u_1,u_2)$ such that $[\varphi y]>[\varphi v_1]$ then 
there is an element
$x\in(u_1,u_2)$ with $ [\varphi x]=\min \big\{[\varphi t]: t\in (u_1,u_2), [\varphi t]>[\varphi v_1]\big\} $. 
Also, if there is $w\in (u_1,u_2)$   such that $[\varphi w]<[\varphi v_2]$, then there is an element $z\in(u_1,u_2)$ with
	$ [\varphi z]=\max \big\{[\varphi w]: w\in (u_1,u_2), \text{ and } [\varphi w]<[\varphi v_2]\big\} $.
\end{axiom}

\begin{lem}\label{lma-min-minus}
	In a model of \Cref{axiom-properties-of-f}, \Cref{ax-min-max}, \Cref{ax-Fib-Functions}, and \Cref{ax-min-decimal-greater-than-phi-c}, if there is $y\in (a,b)$ such that $[\varphi y]> [\varphi c]$ then, 
	the minimum element  $x$ given by the axiom above is obtained by the equation
	$x-c=\min^M_{(a-c,b-c)}$. Similarly, if there is $w\in (a,b)$ such that $[\varphi w]> [\varphi c]$ then, the maximum element $z$ in the axiom above is obtained by the equation
	$z-d=\max^M_{(a-d,b-d)}$.
\end{lem}
\begin{proof}
	As $\fp{x}>\fp{c}$, we have that $\fp{(x-c)}=\fp{x}-\fp{c}$. For any $t\in(a,b)$ consider the element $t-c\in(a-c,b-c)$. If $\fp{t}>\fp{c}$, then $\fp{(t-c)}=\fp{t}-\fp{c}$ which is strictly greater than $\fp{x}-\fp{c}$. 
	On the other hand, if $\fp{t}<\fp{c}$, then $\fp{(t-c)}=\fp{t}-\fp{c}+1$ which is obviously greater than $\fp{x}-\fp{c}$.

For the second part, note that $\fp{(z-c)}=\fp{z}-\fp{d}+1$ as $\fp{z}<\fp{d}$. If $t\in(a,b) $ and $\fp{t}<\fp{d}$, then $\fp{(t-d)}=\fp{t}-\fp{d}+1$ which is strictly smaller than $\fp{z}-\fp{d}+1$. On the other hand, if $\fp{t}>\fp{d}$, then $\fp{(t-d)}=\fp{t}-\fp{d}$ which is strictly smaller than $\fp{z}-\fp{d}+1$.    

\end{proof}
%Having the required preparations, at this point we are able to address one of key difficulties in proving the model-completeness of our theory:

\begin{cor}\label{crl-main-obstacle}
	Assume that $M_1 \subseteq M_2$ are models of \Cref{axiom-properties-of-f}, Axioms \ref{ax-min-max}, \ref{ax-Fib-Functions}, and \ref{ax-min-decimal-greater-than-phi-c} with $a,b,c,d\in M_1$. If $M_2$ satisfies the formula
	\mbox{$\exists x\ \big(a<x<b\wedge [\varphi c]<[\varphi x]<[\varphi d]\big)$},
	then so does $M_1$.
\end{cor}
\begin{proof}
	Using the 
	\Cref{ax-min-decimal-greater-than-phi-c}, there is $x'\in M_2$ with $\fp{c}<\fp{x'}$ such that $ [\varphi x']=\min \big\{[\varphi t]: t\in (a,b), [\varphi t]>[\varphi c]\big\} $. For this element, we clearly have $\fp{x'}<\fp{d}$. By \Cref{lma-min-minus}, $x'$ is equal to $c+\min^{M_2}_{(a-c,b-c)}$. By \Cref{lma-Main}, $\min^{M_2}_{(a-c,b-c)}=\min^{M_1}_{(a-c,b-c)}$, and hence, $x'$ already belongs to $M_1$.
\end{proof}

\begin{rem}
	To better clarify the situation, we emphasize that the argument above does \textit{not} claim that for any model $M$ and any set of elements $a,b,c,d\in M$, the 
	decimal part of the element $ x=c+\min_{(a-c,b-c)} $ falls between the decimal parts of $\fp{c}$ and $\fp{d}$. In fact, there might not exist an element $x\in(a,b)$ with $\fp{x}>\fp{c}$ at all, and in such a case, the decimal part of the element $c+\min_{(a-c,b-c)}$ does not really ly in the decimal interval $(\fp{c},\fp{d})$.  This, in particular, shows why our argument for model-completeness does not extend--at least straightforwardly--to a proof of quantifier elimination even by adding a two-variable function mapping $(a,b)$ to $\min_{(a,b)}$.
\end{rem}

\par
Before proceeding into \Cref{thm-model-completeness}, we need the following lemma which states a set of properties concerning the behaviour of the function $F$:
\begin{lem}\label{lma-about-function-F}
	For all natural numbers $m$ and $n$, the value of $F(m+n)$ can be calculated as the following:
	\[
	\begin{cases}
		f(F(m))+F(m), &\text{if } F(m)\geq F(n) \wedge \big(m+n-(F(m)+F(n))\geq f^{-1}(F(m)-1)\big)\\
		F(m),&\text{if } F(m)\geq F(n) \wedge  \big(m+n-(F(m)+F(n)) <  f^{-1}(F(m)-1)\big)	\\
		f(F(n))+F(n), &\text{if } F(n)>F(m) \wedge \big(m+n-(F(m)-F(n))\geq f^{-1}(F(n)-1)\big)\\
		F(n),&\text{if } F(n)>F(m) \wedge  \big(m+n-(F(m)-F(n)) <  f^{-1}(F(n)-1)\big)	
		%		\\
		%		f(F(m))+F(m) & F(m)=F(n)\wedge \big(m+n-2F(m)\geq f^{-1}(F(m)-1)\big)\\
		%		F(m) & F(m)=F(n)\wedge \big(m+n-2F(m)< f^{-1}(F(m)-1)\big)
	\end{cases},
	\]
	where  $f(F(m))+F(m)$ and  $f^{-1}(F(m)-1)$ are the first Fibonacci numbers with an even index, that are respectively
	greater and less than $F(m)$.
\end{lem}
\begin{proof}
	Suppose that  $F(m)\geq F(n)$ and $m+n-(F(m)+F(n))\geq f^{-1}(F(m)-1)$. In this case, since $F(m) $ and $f^{-1}(F(m)-1) $ are two consecutive Fibonacci numbers, we have $F(m)+F(n)+ f^{-1}(F(m)-1)= f(F(m))+F(m)+f(n)$. Therefore, we have that
	$f(F(m))+F(m)\leq m+n$, which implies $F(m+n)=f(F(m))+F(m)$. 
	\par 
	In the case that the second condition does ot hold, that is $m+n-(F(m)+F(n))< f^{-1}(F(m)-1$, then $m+n<f(F(m))+F(m)$ and hence we have that $F(m+n)=F(m) $.
	\par 
	The argument similarly applies to other cases.
\end{proof}
%\begin{rem}
%	Note that $f(F(x))$ is the first Fibonacci number greater than $F(x)$.
%	Also, we have that $F(f(x))=F(x)$ whenever $F(x)<x<f(F(x))$, and we have $F(f(x))=f(F(x))+F(x)$ whenever $F(F(x))<x<f(F(x))+F(x)$.
%\end{rem}
\begin{axiom}\label{ax-F}
	The statement of the theorem above. 
\end{axiom}

\begin{defn}
	Denote by $T$ the $\mathcal{L}_{\fib}$-theory consisting of
   \Cref{axiom-properties-of-f}, Axioms \ref{ax-min-max}, \ref{ax-Fib-Functions}, \ref{ax-min-decimal-greater-than-phi-c}, and \ref{ax-F}.
\end{defn}

\par 
It is helpful to note that for two models $M_1\subseteq M_2$ of $T$, and for an element $x\in M_2$, the numerical cut of $x $ over $M_1$ 
carries some information about the decimal cut of $\fp{x}$ over $M_1$. More precisely, if $a,b\in M_1$ and $x$ belongs to $(a,b)$, then we have $\fp{ \min^{M_2}_{(a,b)}}<\fp{x}<\fp{ \max^{M_2}_{(a,b)}}$, and this extrema already belongs to $M_1$ by \Cref{lma-Main}. However, the more tightening of the numerical intervals containing $x$, that is the whole numerical cut of $x$ over $M_1$, does not suffice to determine all the decimal information required for our purpose. We will deal with this situation in the course of the proof of \Cref{thm-model-completeness}.

%
%Also, another interesting, yet not-directly-relevant observation can be made here. 
%Assume that $M_1\subseteq M_2$ are two models and $x\in M_2$. Then
%the numerical intervals in which $x$ falls give information about
%the decimal intervals to which it belongs. To be more precise, let $x\in (a,b)$. Then, obviously we have that $[\varphi \min(a,b)]<[\varphi x]<[\varphi \max(a,b)]$.
%But it not the \textit{tightening } of the numerical interval that gives more information. ({\color{red} I just don't precisely understand what this observation amounts to?})

\begin{thm}\label{thm-model-completeness}
	$T$
	is model-complete.
\end{thm}
\begin{proof}
	We need to show that 
	$\mathrm{Diag}(M)\cup T$ is a complete $\mathcal{L}_{\fib}(M)$-theory 
	for each 
	$M\models T$. Suppose that $ M $ is a model 
	of $ T $ with $ |M|<\kappa $ for some $\kappa$, and $M_1 $ and $M_2$ are arbitrary $ \kappa $-saturated
	models of $ T $ with $ M\subseteq M_1 $ and $M\subseteq M_2$. 
	Instead of showing that $ M_1 $ and $ M_2 $ are elementarily equivalent, we prove the stronger
	statement that they are back-and-forth equivalent. 

	Suppose that $ x\in M_1 $, we show that every finite fragment of the quantifier-free type of $ x $ over $ M $ can be satisfied by an element $ y\in M_2 $. According to the properties of the orders involved, we only need to consider quantifier-free formulas of the form below:
	\begin{align*}
		\begin{cases}
			a<x<b\ &\wedge \ \fp{c}<\fp{x}<\fp{d}\ \wedge  \\
			a'<f(x)<b'\ &\wedge \ \fp{c'}<\fp{f(x)}<\fp{d'}\ \wedge  \\
			a''<F(x)<b''\ &\wedge \ \fp{c''}<\fp{F(x)}<\fp{d''}  \\
		\end{cases}
	\end{align*}
	where $a,a',a'', \cdots, d''$ are arbitrarily chosen elements of $M$. Note that according to Axioms (A1.2), (A1.4), \Cref{ax-Fib-Functions}, and \Cref{ax-F} we do not need to consider formulas involving terms of the form $f(F(x)), f(a+F(x)), f^2(x) $ or $F^2(x)$. Also, the numerico-decimal cut of $f(x)$ and $F(x)$ over $M$ is  determined by the numerico-decimal cut of $x$ over $M$. Hence we only need to consider formulas of the form
	 \[ 	a<x<b\ \wedge \ \fp{c}<\fp{x}<\fp{d}. \]
	
	Now, when $M_1$ contains an element satisfying the inequalities above, by \Cref{ax-min-decimal-greater-than-phi-c} there also exists an element satisfying $a<x<b$ and $\fp{c}<\fp{x}$, and having the least possible of corresponding decimal parts. This new element, say $x$, automatically satisfies $ \fp{x}<\fp{d}$ as well. But, \Cref{crl-main-obstacle} ensures that $x$ already belongs to $M$. 
	
	In case that $x$ is infinitely large with respect to elements in $M$, that is, $x$ is numerically greater than every element $a\in M$, finding a corresponding element $y\in M_2$ that satisfies the decimal cut of $x$ over $M$ is always possible due to Kronecker's approximation lemma (\Cref{crl-Kronecker}) and saturation of $M_2$. The same fact remains true for negative elements which are infinitely small since Axiom (A1.3), or less formally, the relation $\fp{(-x)}=1-\fp{x}$ determines the decimal cut of $\fp{(-x)}$ over $M$.

Having $ M\langle x\rangle\cong M\langle y\rangle $, we need to take one more step by extending the substructures $ M\langle x\rangle \subseteq M_1$ and $ M\langle y\rangle \subseteq M_2 $ into models of $ T $, and then by showing that the new extended substructures are  isomorphic too. For this, it suffices to add all the minimum and maximum elements 
	claimed by \Cref{ax-min-decimal-greater-than-phi-c} to the structure $ M\langle x\rangle \subseteq M_1$; theses extrema are already available in $ M_1 $ as a model of $ T $. We do the same for $M\langle y\rangle$ in $M_2$.
	
	To show that the new extended substructures are isomorphic, suppose that $a, b\in M$, and that $t_1(x)$ and $t_2(x)$ are two $\mathcal{L}_{\fib}$-terms. We prove that the elements $\min_{(a+t_1(x),b+t_2(x))}\in M_1$ and $\min_{(a+t_1(y),b+t_2(y))}\in M_2$ satisfy the same 
	set of quantifier-free $\mathcal{L}_{\fib}(M)$-formulas.
	
%	Without loss of generality, we may assume that $ t_1(x)=t_2(x) $ since, ...
	
	\vspace{.5cm}
	\textbf{Case 1.} Suppose that $ t_1(x)=t_2(x)=t(x) $ for some term $t(x)$. First, we prove the following claim:
	
	\textbf{Claim.} The existence of an element $ u\in(a,b)^{M_1} $ with $ \fp{u}>1-\fp{t(x)} $ implies the existence of an element $ v\in (a,b)^{M_2} $ such that $ \fp{v}>1-\fp{t(y)} $. 
	
	\textit{proof of the claim.} As $u\in(a,b)^{M_1}$ and $a,b\in M$, there are arbitrarily large Fibonacci elements in $M$ with a decimal part greater than $\fp{u}$. Hence, we can easily find an element $ d\in M $ such that $ \fp{u}<\fp{d} $. As a part of our argument, we need to repeat the proof of \Cref{crl-main-obstacle}.
	By \Cref{ax-min-decimal-greater-than-phi-c}, there exists an element $ v\in (a,b)^{M_1} $ such that it satisfies $ \fp{v}<\fp{d} $, and moreover it has the largest decimal part
	among all elements having the same properties. By \Cref{lma-min-minus}, the element $ v $ is equal to $ d+\max_{(a-d,b-d)}^{M_1} $, and the latter turns out to be equal to $ d+\max_{(a-d,b-d)}^{M} $; this, 
	by \Cref{lma-Main}. Hence, $ v $ belongs to $ M $. On the other hand, it is obvious that $ \fp{v}\geq\fp{u}>1-\fp{t(x)} $. As $ y $ and $ x $ are assumed to satisfy the same quantifier-free formulas over $ M $, 
	we conclude that $ \fp{v}>1-\fp{t(y)} $ in $ M_2 $.
	\hfill $ \blacksquare_{\text{ Claim}} $

	Now, there are two possible cases for $ \min_{(a+t(x),b+t(x))}^{M_1} $:
	
	\vspace{.5cm}
	\textbf{Case 1.1.} In the case that there does not exist any element $ z\in(a,b)^{M_1} $ with $ \fp{z}>1-\fp{t(x)} $, all the elements $ z\in(a,b)^{M_1} $ have the property that $ \fp{(z+t(x))}=\fp{z}+\fp{t(x)} $. This implies that for any $ w\in(a+t(x), b+t(x)) $ we have that $ \fp{w}>\fp{a} $. Hence, the minimum $ \min_{(a+t(x),b+t(x))}^{M_1} $ is equal to $ t(x)+z_0 $ where $ z_0=\min_{(a,b)}^{M_1} $, and the latter is, by \Cref{lma-Main}, equal to $ \min_{(a,b)}^M $. In addition, $ M_2 $ cannot contain an element  $ z'\in(a,b)^{M_2} $ with $ \fp{z'}>1-\fp{t(y)} $ either. Otherwise, we can use--by way of symmetry--the claim above to show that $ M_1 $ would also contain an element with the same property, which contradicts the hypothesis of this case. Therefore, we can use the same argument to show that $ \min_{(a+t(y),b+t(y))}^{M_2}=t(y)+\min_{(a,b)}^{M_2} $, which, by \Cref{lma-Main}, is equal to $  t(y)+\min_{(a,b)}^{M} $. This, using the fact that $ M\langle x\rangle\cong M\langle y\rangle $, shows that the $ \mathcal{L}_{\fib}(M) $-quantifier-free type of $ \min_{(a+t(x),b+t(x))}^{M_1} $ in $ M_1 $ is the same as that of $ \min_{(a+t(x),b+t(x))}^{M_2} $ in $ M_2 $.
	
	\vspace{.5cm}
	\textbf{Case 1.2.} The second case occurs when there is an element $ z\in(a,b)^{M_1} $ with $ \fp{z}>1-\fp{t(x)} $. Then, by \Cref{ax-min-decimal-greater-than-phi-c}, there exists such an element with the minimum possible of corresponding decimal parts; let $ z_1 $ denote this element. Also, let $ \pi_0(z) \subseteq \tp^{M_1}(z_1/M) $ be the subset consisting of the formulas of the form
	\[ \big\{a'<z<b' \ \wedge \ \fp{c'}<\fp{z}<\fp{d'}\big\}, \]
	for all $ a',b',c',d'\in M $. Note that any finite fragment of the partial type $\pi_0(z)$ reduces to a formula of the form $ a'<z<b' \ \wedge \ \fp{c'}<\fp{z}<\fp{d'} $; let $\theta(z; a',b',c',d')$ denote such a formula.
	
	Let $\pi(z)$ be the union of $\pi_0(z)$ with the the following set of formulas:
	\begin{align*}
		&\Big\{1-\fp{t(y)}<\fp{z}\Big\}\cup\\
		&\Big\{\forall w\Big(\big(\theta(w; a',b',c',d') \wedge \fp{w}>1-\fp{t(y)}\big) \ \rightarrow \ \big(w=z \vee \fp{z}<\fp{w}\big)\Big)\Big\},
	\end{align*}
	where $ \theta(w; a',b',c',d') $ ranges over all formulas in $ \pi_0(z) $.
	
	Now, each $\theta(z; a',b',c',d')$ is satisfied by $z_1$ in $M_1$. Hence, \Cref{ax-min-decimal-greater-than-phi-c} insures the existence of an element $w_0\in M_1$ that satisfies $\theta(z; a',b',c',d')$ and has the maximum possible of decimal parts among all solutions of $ \theta(z; a',b',c',d') $. By \Cref{crl-main-obstacle}, this element, say $ z_0 $, belongs to $ M $. As $\fp{z_1}<\fp{d'}$, we also have $ \fp{z_1}\leq\fp{z_0} $ which implies that $ 1-\fp{t(x)}<\fp{z_0} $. Based on the fact that $ M\langle x\rangle \cong M\langle y\rangle$, the formula $ 1-\fp{t(y)}<\fp{z_0} $ holds in $M_2$ also. 
	
	On the other hand, by \Cref{ax-min-decimal-greater-than-phi-c}, there exists an element, say $z_2\in M_2$ satisfying 
	\[  \theta(z; a',b',c',d')\ \wedge \ 1-\fp{t(y)}<\fp{z} \]
	and with the minimum possible of decimal parts among all solutions of the formula above in $M_2$. 
	
	Hence, any finite fragment of $ \pi(z) $ is satisfied in $ M_2 $. Therefore, we can find a realization for $ \pi(z) $ in $ M_2 $. This shows that isomorphism between $M\langle x\rangle$ and $M\langle y\rangle$ extends to an isomorphism between $\langle M, x,z_1\rangle ^{M_1}$ and $\langle M, y,z_2\rangle ^{M_2}$. 
	
	\vspace{.5cm}
	\textbf{Case 2.} $t_1(x)\neq t_2(x)$. Without loss of generality we may assume that $t_1(x)<t_2(x)$.
	
	Note that $\min^{M_1}_{(a+t_1(x),b+t_2(x))}$ is either equal to $ \min^{M_1}_{(a+t_1(x), b+t_1(x))} $ or equal to $\min^{M_1}_{(b+t_1(x),b+t_2(x))} $. The first case is exactly the same as Case 1 above. Hence, suppose that $\min^{M_1}_{(a+t_1(x),b+t_2(x))} = \min^{M_1}_{(b+t_1(x),b+t_2(x))}$. We show that $\min^{M_1}_{(b+t_1(x),b+t_2(x))}$ and $\min^{M_2}_{(b+t_1(y),b+t_2(y))}$ satisfy the same set of quantifier-free formulas over $M$.
		
	Let $t(x)$ denote $t_2(x)-t_1(x)$. 
	
	\vspace{.5cm}
	\textbf{Case 2.1.} If there is no element $z\in(0,t(x))$ with $\fp{z}>1-\fp{b}$, we have $\fp{(b+z)}=\fp{b}+\fp{z}$ for all elements $z\in(0,t(x))$. Hence, $\min^{M_1}_{(b+t_1(x),b+t_2(x))}$ is equal to $b+F(t(x))$ as $F(t(x))$ has the minimum of corresponding decimal parts among all elements in $(0,t(x))$. But, a similar argument as above shows that $M_2$ cannot contain such an element either. That is, $\min^{M_2}_{(b+t_1(y),b+t_2(y))}=b+F(t(y))$ which belongs to $M\langle y\rangle $. By our assumption of $M\langle x\rangle\cong M\langle y\rangle$, this implies that $\min^{M_1}_{(b+t_1(x),b+t_2(x))}$ and $\min^{M_2}_{(b+t_1(y),b+t_2(y))}$ satisfy the same set of quantifier-free formulas over $M$.
	
	\vspace{.5cm}
	\textbf{Case 2.2.} For the case that there exists an element $z\in(0,t(x))$ with $\fp{z}>1-\fp{b}$, note that, by \Cref{ax-min-decimal-greater-than-phi-c}, there also exists an element $z_1 \in(0,t(x))$ in $M_1$ satisfying $\fp{z_1}>1-\fp{b}$ and with the minimum corresponding decimal part among all elements like $z$. Note that, we have $\fp{(b+z_1)}=\fp{b}+\fp{z_1}-1$, and hence $\min^{M_1}_{(b+t_1(x),b+t_2(x))}$ would be equal to $b+z_1$. 
	
	\vspace{.5cm}
	\textbf{Case 2.2.1} Suppose that there exists a positive element $a'\in M$ with $t(x)-a'<z_1<t(x)$. Then, there is an element $z'\in(0,a')^{M_1}$ such that $z_1=t(x)-z'$. Recall that $z_1\in M_1$ is the element having the least decimal part among all elements satisfying 
	\begin{align}\label{eq-z-1}
			 1-\fp{b}<\fp{z_1}\ \wedge \ 0<z_1<t(x).
	\end{align}

	There are two possible ways that $\fp{z_1}$ and $\fp{z'}$ can interrelate:
	\begin{align*}
		\begin{cases}
			\text{(ii)} \hspace*{7pt}\fp{z_1}=\fp{t(x)}-\fp{z'} &\text{ when } \fp{z'}<\fp{t(x)}, \text{ or }\\
			\text{(i)} \hspace*{3pt} \fp{z_1}=\fp{t(x)}-\fp{z'}+1 &\text{ when } \fp{z'}>\fp{t(x)}.
		\end{cases}
	\end{align*}
	Either way, in order for $z_1$ to have the least decimal part among all elements having the properties \eqref{eq-z-1} above, the element $z'$ must have the greatest decimal part among all 
	elements satisfying 
	\begin{align}\label{eq-z-prime}
		 0<z'<a' \ \wedge\  z_1=t(x)-z' \ \wedge\ 1-\fp{b}<\fp{(t(x)-z')}.
	\end{align}
	 In case (i), conditions \eqref{eq-z-1} and \eqref{eq-z-prime} translate into having $z'\in M_1$ with the greatest possible decimal part and satisfying 
	 \[ 0<z'<a'\ \wedge \ \fp{t(x)}<\fp{z'}<\fp{t(x)}-\fp{b}, \]
	 where $\fp{t(x)}-\fp{b}$ has no other choice than being positive. That is, it equals $\fp{(t(x)-b)}$, and therefore, case (i) reduces to having the element $ z'\in M_1$ that has the greatest decimal part among all elements satisfying $ 0<z'<a'$ and $ \fp{z'}<\fp{(t(x)-b)}$. Here, note that the very existence of $z'$, guaranteed by the existence of $z_1$, ensures that the condition $ \fp{t(x)}<\fp{z'}$ is already satisfied as $z'$ has the greatest possible of decimal parts. On the other hand, $t(x)-b$ belongs to $M\langle x\rangle$, meaning that, it is of the form $t'(x)$ for some $\mathcal{L}_{\fib}$-term $t'(x)$. Hence, this case can be handled similar to Case 1. 
	 
	 As for case (ii), the conditions \eqref{eq-z-1} and \eqref{eq-z-prime} translate into having $z'\in M_1$ satisfying 
	 \[ 0<z'<a'\ \wedge \ \fp{z'}<\fp{t(x)}\ \wedge \ \fp{z'}<\fp{t(x)}+\fp{b}-1, \]
	 and having the greatest possible decimal part. As $\fp{z'}$ is positive, the last inequality urges that $\fp{t(x)}+\fp{b}$ be greater than one, or equivalently, that $\fp{(t(x)+b)}=\fp{t(x)}+\fp{b}-1$. ‌Again, $t(x)+b$ belongs to $M\langle x\rangle$. Now, whether $\fp{t(x)}$ is greater or less than $\fp{(t(x)-b)}$ remains unchanged for $t(y)$ and $b$ in $M_2$, and therefore, we can use a similar argument as used in Case 1 to find the desired element in $M_2$. 
	
	\vspace{.5cm}
	\textbf{Case 2.2.2} There does not exist any positive element $a'\in M$ with $t(x)-a'<z_1<t(x)$. Then, we have $a'<z_1+a'<t(x)$ for all $a'\in M$, that is, $t(x)$ is infinitely large with respect to the elements in $M$, and the same remains true for $t(y)$ in $M_2$. 
	
	Now, one can see that the element $F(t(y))\in M_2$ is also greater than any element in $M$. Otherwise, there is an element $a'\in M$ such that $F(t(y))<a'<t(y)$. Recall that we already have $F(t(y))<t(y)<F(t(y))+f(F(t(y)))$ where the latter element is the first Fibonacci element greater than $t(y)$ and of an even index. As $f$ is an increasing function, we will have that $F(t(y))+f(F(t(y)))< a'+f(a')$ implying that $t(y)<a'+f(a')$. But the latter element belongs to $M$ contradiction our hypothesis. 

    Note that $G(t(y))=f(F(t(y)))$ which enables us to similarly argue that $G(t(y))$ is also greater than any element in $M$. In particular, we have $G(b)<G(t(y))$. 
    Note that we have $1-\fp{b}<1-\fp{F(b)}<\fp{G(b)}$, and 
    this is true for any element in any model of $T$ by the way. In addition, $G(b)<G(t(y))$ implies that $\fp{G(b)}<\fp{G(t(y))}$. That is, for $G(t(y))\in M_2$‌ we have that 
    \[ 0<G(t(y))<t(y) \text{ and } 1-\fp{b}<\fp{G(t(y))}\]
    which enables us to find an element in $M_2$ having the properties above and with the least possible decimal part. Hence, we can proceed as in Case 1.2 above.
    
    The extrema addressed in \Cref{ax-min-decimal-greater-than-phi-c} can be handled using similar arguments.

\end{proof}
\begin{cor}
	$T$ is decidable and axiomatizes the structure ${\zphi=\ZphiF}$. 
\end{cor}			
\begin{proof}
	$T $ is complete, because
	it is model-complete and has  $\ZphiF$ as its prime model. The axiomatization is clearly recursive, and hence the theory is decidable. 
\end{proof}	

%\begin{rem}
%	Based on our argument for model-completeness of the theory $T$, it seems plausible to obtain a quantifier elimination by 
%	adding two more function symbols which capture the extrema mentioned in \Cref{lma-min-minus}.
%\end{rem}

\section{Final remarks}
%We would like to conclude the article by presenting some final remarks, which we find worthy of attention. 

\subsection{An interrelated structure} The plot of the function
$[\cdot]:\mathbb{N}\to \mathbb{R}$, with 
the rule
$x\mapsto [\varphi x]$ can be depicted using the black points appearing in the following diagram (for $n\leq 60$):
\begin{figure}[h]
	\centering
	\includegraphics[width=13cm]{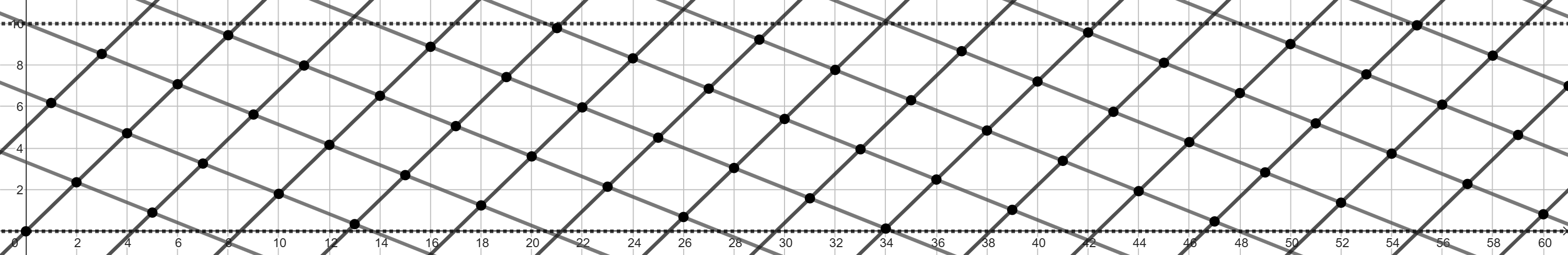}
\end{figure} 

As can be seen above, the black nodes are the intersections of equidistant parallel lines with positive and negative (calculable) slope, and contains seemingly simple patterns. 
Let us mean by the \textit{diamond structure} the 
structure
obtained by considering the points of intersection 
of such line segments vertically confined in the interval $(0,1)$. 
One can ask the following seemingly simple question:
\begin{quote}
	Axiomatize the diamond structure in a suitable language and prove that this structure is decidable.
\end{quote}
As the relation $[\varphi x]<[\varphi y]$ is definable in our theory, the diamond diagram is
hidden in our models, and the decidability of its structure follows our theorem. 
\par 
However, and quite interestingly, the diamond diagram will look much alike when drawn for the function
$[ ex]$ for $e$ denoting the Euler number. To be more precise, in this case the slope of the mentioned lines would be equal to $e$ or $-e$
instead of $\varphi$ and $-\varphi$. In other words, the only thing that changes is the number of intersecting nodes 
lying on the line segments. 

Having a better understanding of these diamond structures may provide reliable information on
the complexity of the numerical structure obtained by the function $\lfloor ex\rfloor$. As we guess, this can provide a surprising
shortcut probably to obtain a similar decidability result for the structure  $\langle \mathbbm{Z},<, +,f_e,0\rangle$ where $f_e(x)=\lfloor ex\rfloor$ for each integer $x$. See also the Section ``A connection to o-minimality'' in \cite{kvz}.

\subsection{An Algorithm}\label{subsec-algorithm}
The proof of the proposition below shows that the smallest corresponding decimal part among all integers between two given integers $a$ and $b$ can be found algorithmically. Similar algorithm can be applied to find the desired minimum for negative integers as well. One may expect the same algorithm to work also in non-standard models.
However, this is not necessarily the case mainly due to the fact that there does not exist the least non-standard Fibonacci number, and as a consequence, not all non-standard elements admit a finite Fibonacci representation. 

%
% question of finding an element in 
%an arbitrary interval $(a,b)$ with the smallest corresponding decimal part, 
%We now return to the more general question of finding an element in 
%an arbitrary interval $(a,b)$ with the smallest decimal part.
%Let $\min^{M}_{(a,b)}$ be the element $x\in (a,b)$ such that
%$[\varphi x]=\min \{[\varphi t]:t\in (a,b)\}$.
%There is an algorithm to find
%$\min^{\mathbbm{N}}_{(a,b)}$ for $a,b\in \mathbbm{N}$, as follows.
\begin{prop}\label{prop-algorithm}
	Let $a$ and $b$ be two integers with $0<a<b$. The element $c$ with $a<c<b $ and $\fp{c} = \min \big\{\fp{t}:a<t<b\big\} $ can be obtained 
	using the following procedure: 
	Let $n=0$, $a_0=a$ and $b_0=b$.
	\begin{enumerate}
		\item 
		If $F(b_n)\leq a_n$, it is implied that $F({b_n})=F({a_n})$. Put $a_{n+1}=a-F({a_n})$ and
		$b_{n+1}=b-F({b_n})$.
		
		\item If 
		$F(b_n)>a_n$, let $c=F({b_n})+\displaystyle \sum_{i=0}^{n-1} F({a_i})$ and halt.
		
		\item 
		Replace $n$ with $n+1$ and go to step (1).
	\end{enumerate}
\end{prop}
\begin{proof}
	First, note that the sequence $\{\fp{F_{2n}}\}_{n\in \NN}$ is a strictly decreasing sequence. Hence, if 
	$F(b)>F(a)$, by Lemma \ref{lma-fib-min-max}, the value of $\fp{F(b)}$ is the desired element with 
	the minimum of corresponding decimal parts. 
	\par 
	If $F(b)>F(a)$, it is sufficient to show that if $F(b-F(a))>a-F(a)$, then $\min \big\{\fp{t}:a<t<b\big\} =\fp{\big(F(a)+F(b-F(a))\big)}$.
	Since $F(b)=F(a)$, both $a$ and $b$ lie between two consecutive Fibonacci number with an even index, that is 
	$F(a)<a<b<f(F(a))+F(a)$, where $f(F(a))+F(a)$ is the first Fibonacci number greater than $F(a)$ and with an even index.
	Note that every $F(a)<x<f(F(a))+F(a)$, can be written as $F(a)+y$, for some $0<y<f(F(a))$.
	\par 
	On the other hand, by the proof of Lemma \ref{lma-fib-min-max}, 
	for all $0<y<f(F(a))$, we have $\fp{(F(a)+y)}=\fp{F(a)}+\fp{y}$. Hence if $F$ is the greatest Fibonacci
	number with an even index satisfying  $a<F(a)+F<b$, then $\fp{(F(a)+F)}$ attains the desired minimum. But, 
	$F=F(b-F(a))$, so $\min \big\{\fp{t}:a<t<b\big\} =\fp{\big(F(a)+F(b-F(a))\big)}$.
	
	This procedure ultimately halts as the obtained sequence of Fibonacci numbers of an even index decreasingly approaches zero. 
\end{proof}
%\subsection{Remark 2.}
%As appeared in \Cref{crl-fib-definalbe}, the set of Fibonacci numbers is definable using an $\mathcal{L}$-formula which is $\Pi_1$. 
%Although finding a quantifier-free formula defining the Fibonacci sequence would be surprisingly interesting, we cannot see an easy way of finding such a formula!
%
%%
%\subsection{Remark 3.}
%Ali's explanation of the Fibonacci hull. 

\end{document}